\documentclass[12pt]{amsart}
\usepackage{amsfonts}
\usepackage{amsmath}
\usepackage{amssymb}
\usepackage{graphicx}
\usepackage{amscd}
\usepackage{xcolor}

\usepackage{alphabeta}
\usepackage{ulem}
\usepackage{bbm}

\allowdisplaybreaks

\usepackage{hyperref}
\newtheorem{theorem}{Theorem}

\newtheorem{lemma}[theorem]{Lemma}

\newtheorem{proposition}[theorem]{Proposition}

\begin{document}
	
	\begin{abstract}
		We prove that for a certain class of $n$ dimensional rank one  locally symmetric spaces, if $f \in L^p$, $1\leq p \leq 2$, then the Riesz means of order $z$ of $f$ converge to $f$ almost everywhere, for $\mathrm{Re}z> (n-1)(1/p-1/2)$.
	\end{abstract}

	\title[Riesz means]{Riesz means on locally symmetric spaces}
	\author{
		Effie Papageorgiou 
	}
	\title{Riesz means on   locally symmetric spaces}
	\thanks{Supported by the Hellenic Foundation for Research and Innovation, Project HFRI-FM17-1733.}
	\email{papageoeffie@gmail.com}
	\address{Department of Mathematics, Aristotle University of Thessaloniki,
		Thessaloniki 54124, Greece.}
	\curraddr{Department of Mathematics and Applied Mathematics, University of Crete, Voutes Campus, 70013 Heraklion, Crete, Greece.}

	\maketitle
	
	\section{Introduction and statement of the results}
	
	In this article we study the almost everywhere convergence of the Riesz means on quotients of rank one noncompact symmetric spaces over a proper subgroup of isometries. Recall that noncompact rank one symmetric spaces are the real, complex and quaternionic hyperbolic spaces,  and the octonionic hyperbolic plane.  We extend the results obtained on rank one symmetric spaces in \cite{GIUMA} to a class of rank one locally symmetric spaces. To state our results, we need to introduce some notation.
	
	Let $G$ be a semi-simple, noncompact, connected Lie group with finite
	center and let $K$ be a maximal compact subgroup of $G$. Consider the
	symmetric space of noncompact type $X=G/K$. Let $\dim X=n$. Denote by $\mathfrak{g}$ and $\mathfrak{k}$ the Lie algebras of $G$ and $K$, respectively. We have the Cartan decomposition $\mathfrak{g}=\mathfrak{p} \oplus  \mathfrak{k}$. Let $\mathfrak{a}$ be a maximal abelian subspace of $\mathfrak{p}$ and let $\mathfrak{a}^*$ be its dual.  If $\mathfrak{a} \cong \mathbb{R}$, then we say that $X$ has rank one. From now on, we assume that $\text{rank}X=1$. 
	
	Denote by $\rho $ the half sum of positive roots counted with their
	multiplicities, which is just a positive constant depending on the structure of $X$. Fix $R\geq  \rho^{2}$ and $z\in \mathbb{C}$ with
	$\mathrm{Re}z>0 $, and consider the even and bounded non-negative function
	\begin{equation}  \label{mult}
		s_{R}^z(\lambda)=\left( 1-\frac{\lambda^2+\rho^2}{R}\right)^z_{+}, \; \lambda \in \mathfrak{a}^{\ast}\cong \mathbb{R}.
	\end{equation}
	Denote by $\kappa_{R}^z$ the inverse spherical Fourier transform of 
	$s_{R}^z$ in the sense of distributions and consider the Riesz means on $X$ to be the convolution operator $S_R^z$: 
	\begin{equation}
		S_R^z(f)(x)=\int_{G}\kappa _R^z(y^{-1}x)f(y)dy,\quad f\in C_{0}^{\infty}(X).  \label{operatorX}
	\end{equation}
	
	Let $\Gamma $ be a discrete and torsion free subgroup of $G$ and consider the locally symmetric space $M=\Gamma \backslash X=\Gamma
	\backslash G/K.$ Then $M$, equipped with the projection of the canonical
	Riemannian structure of $X$, becomes a Riemannian manifold.
	
	To define Riesz means on $M$, we first observe that if $f\in
	C_{0}^{\infty }(M)$, then the function $S_{R}^{z }f$ defined by (\ref{operatorX}) is right $K$-invariant and left $\Gamma $-invariant. So
	$S_{R}^{z }$ can be considered as an operator acting on functions on $M,$ which we shall denote by $\widehat{S}_{R}^{z }$.
	
	Denote by $d(\cdot, \cdot)$ the Riemannian distance on $X$ and let
	\begin{equation}\label{poincareseries}
		P_s(x, y)=\sum_{ \gamma \in \Gamma}e^{-sd(x,\gamma y)}, \quad \forall s>0, \,\forall x, y\in X,
	\end{equation}
	be the Poincar\'{e} series. Note that $P_s(\cdot, \cdot)$ can be both viewed as a function on $X\times X$ as on $M \times M$.
	The critical exponent $\delta(\Gamma)$ is defined by
	\begin{equation*}
		\delta (\Gamma)=\inf \{ s>0: P_s(x,y)< +\infty \},
	\end{equation*}
	and is independent of the choice of $x,y \in X$. It may also be defined by
	\[
	\delta(\Gamma)= \limsup_{R\rightarrow +\infty}\frac{\log N_R(x,y)}{R}, \quad \forall x, y \in X,
	\]
	where $N_R(x, y) = \#\{\gamma \in \Gamma|\; d(x, \gamma y) \leq R\}$ denotes the orbital counting function.
	Using the fact that in rank one symmetric spaces, the volume growth at infinity is exponential, i.e., for all $x\in X$ and $R>1$, $|B(x,R)|\asymp e^{2\rho R}$, \cite{Kni97}, it can be shown that  $\delta(\Gamma)\in[0,2\rho]$, see \cite[Section 1.6]{Nicholls} for real hyperbolic space, generalized for arbitrary rank symmetric spaces in \cite[Section 2]{Leuz2004}.

	We say that a rank one locally symmetric space $M = \Gamma \backslash G/K$ belongs in the \textit{class (R)} if 
	\begin{itemize}
		\item[(i)]  $\delta(\Gamma)< \rho$,
		\item[(ii)] $\sup_{x, y \in X}P_s(x, y)<+\infty$, for any $s>\delta(\Gamma)$, and
		\item[(iii)] $M$ has bounded geometry.
	\end{itemize}
	
We say that a manifold $M$ has bounded geometry if its Ricci curvature is uniformly bounded below (always true for complete locally symmetric spaces) and if its injectivity radius is bounded away from zero. Conditions (ii) and (iii) hold for instance for convex cocompact groups. Recall that $\Gamma$ is called convex cocompact if $\Gamma \backslash \text{Conv}(\Lambda_{\Gamma})$ is compact, where $\text{Conv}(\Lambda_{\Gamma}$) be the convex hull of the limit set $\Lambda_{\Gamma}$ of $\Gamma$. 
	
	Our main result is the following, which  extends the results of \cite{GIUMA} for rank one noncompact symmetric spaces, in the class (R) of rank one locally symmetric spaces.

	\begin{theorem}\label{Thm}
		Let $1\leq p \leq 2$. If $M \in (R)$ and $\mathrm{Re}z>(n-1)(\frac{1}{p}-\frac{1}{2})$, then 
		\begin{equation}\label{aec}\
			\lim\limits_{R\rightarrow +\infty}\widehat{S}_R^zf(x)=f(x), \; \text{a.e., for } f\in L^p(M).
		\end{equation}
	\end{theorem}
	
	The Riesz means operator has been extensively studied in the case of $\textbf{R}^n$ (\cite{CHRIST8,CHRIST7,DAVCHANG,STEIN29}). The case of Lie groups and Riemannian manifolds of non-negative curvature is treated in \cite{ALEXOLO} and the case of elliptic differential operators on compact manifolds in \cite{Berard,ChristSogge,GIUTRA,HORM,Sogge}. For the case of rank one noncompact symmetric spaces see \cite{GIUMA}, for $SL(3, \mathbb{H})/Sp(3)$ see \cite{zhu} and for a general result on symmetric spaces (although not optimal)  see \cite{RieszFP}.  Note that as in the euclidean case of $\mathbb{R}^{n}$, \cite{STEIN29}, as well as in
	rank one symmetric spaces, \cite{GIUMA}, we obtain that (\ref{aec}) is valid
	for $\mathrm{Re}z$ larger than the critical index $z_0(n,p)= (n-1) \left( \frac{1}{p}-\frac{1}{2}\right)$.
	
	Unlike the euclidean case, or for instance Riemannian manifolds of non-negative
	curvature, \cite{ALEXOLO}, where (\ref{aec}) follows from $L^p$ continuity of a Riesz means maximal operator,  even the problem of $L^p$ boundedness of $S_R^z$, $p\neq 2$, is ill posed on noncompact symmetric spaces. Indeed, for rank one, $s_R^z$ is compactly supported so it does not extend to a holomorphic function in any complex strip containing the real line, an extension that would be a necessary condition, \cite{CLERC}. 
	
	Instead, one could use $L^p \rightarrow L^r$ mapping properties of the  maximal operator $S_{\ast}^{z}f(x)=\sup_{R \geq\rho^2}|S_{R}^{z}f(x)|$, 
	based on estimates of the kernel $\kappa_R^z$. More precisely, the approach pursued in \cite{GIUMA} for rank one symmetric spaces, is to show that if $\text{Re}z>(n-1)/2$, then $S_{\ast}^{z}f(x)\leq M_1f(x)+c |f|\ast k(x)$, where $M_1$ is the Hardy-Littlewood maximal function over the balls of radius less than $1$, and $k$ is a kernel in $L^q$ for every $q \geq 2$. The required almost everywhere convergence follows from  mapping properties of $M_1$ and $\ast \kappa$, and complex interpolation. 
	
	In the present setting which deals with the class (R) of locally symmetric spaces, condition (i) allows firstly to transform estimates of $\kappa_{R}^z$ on $X$ to kernel bounds on $M$, thus defining an integral Riesz means operator $\widehat{S}_R^z$ on $M$. Next, we decompose the corresponding maximal operator $\widehat{S}_{\ast}^z$ to a local part and a part at infinity. The assumption of bounded geometry (iii) (and conditions (i), (ii), via a heat kernel argument) ensures that all balls of the same small radius behave like their euclidean counterparts. Thus, the local part of $\widehat{S}_{\ast}^z$ can be also controlled by the Hardy-Littlewood maximal operator. The part at infinity can be viewed as an integral operator on $M$, with a kernel belonging to all  $L^q(M)$, $q\geq q_0(M)$, where $q_0(M)$ is large enough, but finite. Assumption (ii) on the uniform bound of Poincar{\'e} series turns out to be crucial in both parts.

	This paper is organized as follows. In Section 2 we present the necessary
	preliminaries, and in Section 3 we prove Theorem \ref{Thm}. 
	
	Throughout this article, the different constants will always be denoted by the same letter $c$.
	
	\section{Preliminaries}
	In this section we recall some basic facts about symmetric spaces. For more
	details see for example \cite{HEL, HELnew}.
	\subsection{Symmetric spaces}
	Let $G$ be a semisimple Lie group, connected, noncompact, with finite center
	and let $K$ be a maximal compact subgroup of $G$. We denote by $X$ the noncompact symmetric space $G/K$. In the sequel we
	assume that $\mathrm{dim}X=n$. Denote by $\mathfrak{g}$ and $\mathfrak{k}$ the Lie
	algebras of $G$ and $K$. Let also $\mathfrak{p}$ be the subspace of $%
	\mathfrak{g}$ which is orthogonal to $\mathfrak{k}$ with respect to the
	Killing form. The Killing form induces a $K$-invariant scalar product on $%
	\mathfrak{p}$ and hence a $G$-invariant metric on $G/K$. Denote  by $d(.,.)$ the Riemannian
	distance and by $dx$ the associated Riemannian measure on $X$. 
	
	Fix $\mathfrak{a}
	$ a maximal abelian subspace of $\mathfrak{p}$ and denote by $\mathfrak{a}%
	^{\ast }$ the real dual of $\mathfrak{a}$. If $\mathrm{dim}\mathfrak{a}=l$, we
	say that $X$ has rank $l$. We also say that $\alpha \in \mathfrak{a}^{\ast }$ is a root vector, if  
	\begin{equation*}
		\mathfrak{g}^{\alpha }=\left\{ X\in \mathfrak{g}:[H,X]=\alpha (H)X,\text{
			for all }H\in \mathfrak{a}\right\} \neq \left\{ 0\right\} .
	\end{equation*} Denote by $\rho$ the half sum of positive roots, counted with their multiplicities.

	From now on, we assume that $\mathrm{rank}X=1$. Then, $X$
	is one of the following: real hyperbolic space $H^n(\mathbb{R})$, complex hyperbolic space $H^n(\mathbb{C})$,
	quaternionic hyperbolic space $H^n(\mathbb{H})$ or the octonionic hyperbolic plane $H^2(\mathbb{O})$. The constant $\rho$ is, respectively, $(n-1)/2$, $n$, $2n+1$ and $11$. We have the Cartan decomposition 
	\begin{equation}
		G=K\exp \overline{\mathfrak{a}^{+}}K,  \label{kak}
	\end{equation}
	where $\overline{\mathfrak{a}^{+}}\cong[0, +\infty)$. On $X=G/K$, the decomposition $K\exp \overline{\mathfrak{a}^{+}}$ corresponds to polar coordinates. Therefore, each element $g \in G $ is written as $g =k(\exp H)k^{\prime}$, where the component $H\geq 0$ is unique. Define $|g|=H$. Viewed on $X=G/K$, $|g|$ is the distance $d(x, o)$ of $x=gK$ to the origin $o=K$. 
	
	We identify functions on $X = G/K$ with functions on $G$ which are $K$-invariant on the right, and hence bi-$K$-invariant functions on $G$ with functions on $X$ that are $K$-invariant on the left. In the rank one setting, that simply means radial. 
	
	\subsection{The spherical Fourier transform}
	
	Denote by $S(X)^{\#}$ the Schwartz space of radial functions on $X$. The spherical Fourier transform $\mathcal{H}$ is defined
	by 
	\begin{equation*}
		(\mathcal{H}f)(\lambda )=\int_{G}f(x)\varphi _{\lambda }(x)\;dx,\quad
		\lambda \in \mathbb{R},\quad f\in S(X)^{\#},
	\end{equation*}
	where $\varphi _{\lambda }$ are the elementary spherical functions on $G$. 
	Let $S(\mathbb{R})$ be the usual Schwartz space on the real line and let $S(\mathbb{R})^{even}$ be the subspace
	of even functions in $S(\mathbb{R})$. Then, by a
	celebrated theorem of Harish-Chandra, $\mathcal{H}$ is an isomorphism
	between $S(X)^{\#}$ and $S(\mathbb{R})^{even}$. Its inverse
	is given by 
	\begin{equation*}
		(\mathcal{H}^{-1}f)(x)=\text{const.}\int_{\mathbb{R} }f(\lambda )\varphi
		_{-\lambda }(x)\frac{d\lambda }{|\mathbf{c}(\lambda )|^{2}},\quad x\in
		G,\quad f\in  S(\mathbb{R})^{even},
	\end{equation*}
	where $\mathbf{c}(\lambda )$ is the Harish-Chandra function.

	\subsection{The class (R)} As mentioned earlier, the present setting concerns the class (R) of locally symmetric spaces $M=\Gamma \backslash X=\Gamma
	\backslash G/K$, where (i) the critical exponent $\delta(\Gamma)<\rho$, (ii) Poincar{\'e} series is uniformly bounded for any $s>\delta(\Gamma)$, and (iii) $M$ has bounded geometry. In this subsection, we recall some known results concerning the assumptions above.
	
	It is well known by various results by Elstrodt, Patterson, Sullivan and Corlette that when $\delta(\Gamma)<\rho$, the
	bottom of the $L^2$-spectrum of the Laplace-Beltrami operator $-\Delta_M$ on $M$ is equal to $ρ^2$, as on $X$, see for instance \cite{Leuz2004} and the references therein. The latter implies that $M$ is of infinite volume, \cite{JLW}.
	
	Poincar{\'e} series, thus its pointwise estimates, arises naturally in the study of locally symmetric spaces. For the class (R), we  require a uniform upper bound, 
	\begin{equation}\label{poincare_Cocompact}
		\sup_{x, y \in X}P_s(x,y) \leq C(s) <+\infty, \quad \forall s>\delta(\Gamma).
	\end{equation}
Such a uniform upper bound is true, for instance, in the case of convex cocompact groups. More precisely, in \cite[Lemma 3.3]{Zhang18} it was proved that if $\Gamma$ is convex cocompact, then there exists a constant $C > 0$ such that for all $x, y\in X$, and every $s>\delta(\Gamma)$, it holds $P_s(x,y) \leq C\, P_s(o, o)$,
	where $o = eK$ denotes the origin in $X$. The subject of convex cocompact groups has been of extensive study. Without any intention to exhaust the vast literature, we refer to \cite{COR1, Nicholls, Yue, Zhang18} for associated characterizations and interesting results.  
	
	Finally, let us comment on the assumption of bounded geometry. To begin with, recall that the standard notion of injectivity radius on Riemannian manifolds, in the case of $M=\Gamma \backslash X$ boils down to $\text{inj}(M) = \inf_{\tilde{x}\in M}\text{inj}(\tilde{x})$, where
	\[
	\text{inj}(\tilde{x})=\frac{1}{2}\inf\{d(x, \gamma x):\; \gamma \in \Gamma \smallsetminus\{\text{id}\}, \; x\in \pi^{-1}(\tilde{x})\}.
	\]
	Here, $\pi: X \rightarrow \Gamma\backslash X$ denotes the canonical projection. It follows that $\Gamma$ cannot contain parabolic elements. For convex cocompact groups, the injectivity radius is bounded away from zero (this was pointed out to us by J.-Ph. Anker). 
	Furthermore, bounded geometry implies some control on the volume growth of $M$.  For an upper bound, recall that a ball of radius $r$ on $M$ has volume less or equal than the corresponding ball on its universal cover $X$. It follows, using the Cartan decomposition (or see \cite[p.647]{ANDAYA}), that there is a constant $c>0$ such that for all $\tilde{x}\in M$, $r>0$,
	\begin{equation}\label{volume_annulus} 
		|B(\tilde{x},r)| \leq c\, r^{n}e^{2\rho r}.
	\end{equation} For a lower bound concerning volume growth for $M$ in the class (R), we make use of the Faber-Krahn inequality.
	\subsubsection{Faber-Krahn inequality}  
	Given a non-negative non-increasing function $\Lambda$ on $(0, +\infty)$, we say
	that a Riemannian manifold $M$ satisfies the \textit{Faber-Krahn inequality with
		function $\Lambda$} if, for any non-empty relatively compact open set $\Omega\subset M$, 
	\[
	\lambda_{\min}(\Omega)\geq \Lambda(|\Omega|), 
	\]
	where $\lambda_{\min}(\Omega)$ is the first eigenvalue of the Dirichlet problem in $\Omega$ for the Laplace-Beltrami operator $-\Delta_M$
	and $|\Omega|=\text{vol}(\Omega)$, \cite[Section 14.2]{GrigBook}. 
	
	Denote by $h_t$ the heat kernel on a Riemannian manifold. Then, for any $\nu>0$, the following conditions are equivalent:
	\begin{itemize}
		\item[(a)]\label{(a)} The on-diagonal estimate $h_t (x, x) \leq C\, t^{-\frac{\nu}{2}}$, for all $t > 0$ and $x\in M$.
		\item[(b)] The Faber-Krahn inequality with function $\Lambda(v)=c\,v^{-\frac{2}{\nu}}$ where $c>0$,
	\end{itemize}
	\cite[Corollary 14.23]{GrigBook}. Note that a Faber-Krahn inequality with function
	$\Lambda(v)=c\,v^{-\frac{2}{\nu}}$
	implies that, for any relatively compact ball $B(x, r)$,
	\begin{equation}\label{lower_vol}
		|B(x,r)|\geq c(\nu)r^{\nu},
	\end{equation}
	\cite[p.371]{GrigBook}. In the case of bounded geometry, geodesic balls are indeed relatively compact sets, see \cite[pp.312-313]{GrigBook}. 
	
	There is a very rich and long literature concerning heat kernel estimates in various geometric contexts, \cite{GrigBook}. In particular, optimal estimates of the heat kernel have been obtained in \cite{DAVMAN} for real hyperbolic spaces, generalized in  \cite{ANDAYA} for Damek-Ricci spaces (which include all rank one noncompact symmetric spaces) and finally in \cite{ANJ} and  \cite{ANO}
	for arbitrary rank noncompact symmetric spaces.
	
	More precisely, the heat kernel $h_t$ on rank one symmetric spaces is a radial function, that is $h_t(x,y)=h_t(d)$, where $d=d(x,y)$ is the geodesic distance on $X$. The following (upper and lower) estimate holds:
	\begin{equation} \label{heat_est1}
		h_t(d) \asymp t^{-\frac{3}{2}}(1+d)\left(1+\frac{1+d}{t}\right)^{\frac{n-3}{2}} e^{-\rho^2 t-\rho d-\frac{d^2}{4t}}.
	\end{equation}
	Note that \eqref{heat_est1} implies the upper bound 
	\begin{equation}\label{heat_est2}
		h_t(d)\leq c(\varepsilon)\, t^{-\frac{n}{2}}e^{-(\rho-\varepsilon)d} \quad \text{ for all } t, d>0,
	\end{equation}
	for any $0<\varepsilon<\rho$.
	
	The heat kernel $h_t^M$ on $M=\Gamma \backslash X$ is given by
	\[
	h_t^M(\tilde{x}, \tilde{y})=\sum_{\gamma \in \Gamma }h_t(x, \gamma y),
	\] 
	\cite{COR1}. Therefore, when $\delta(\Gamma)<\rho$, we have by (\ref{poincare_Cocompact}) and (\ref{heat_est2}) that
	\begin{equation}
		h_t^M(\tilde{x}, \tilde{y}) \leq c(\varepsilon)\, t^{-\frac{n}{2}} \sum_{\gamma \in \Gamma }e^{-(\rho-\varepsilon)d(x, \gamma y)} \notag\\
		\leq c(\varepsilon, \rho)\, t^{-\frac{n}{2}}, \label{FKcond}
	\end{equation}
	taking $\varepsilon$ small enough so that $\rho-\varepsilon>\delta(\Gamma)$.
	
	Thus, according to the equivalent conditions (a) and (b),  the locally symmetric space $M=\Gamma\backslash X$ in the class (R) satisfies the Faber-Krahn inequality with $\Lambda(v)=c\, v^{-\frac{2}{n}}$. Hence, the lower bound  (\ref{lower_vol}) holds for $\nu=n$. 
	
	It follows from (\ref{volume_annulus}) and (\ref{lower_vol}) for $\nu=n$ that, for locally symmetric spaces in the class (R), small balls are essentially euclidean:
	\begin{equation}\label{small_balls}
		|B_M(\tilde{x},r)|\asymp r^n, \quad \text{ for all } \tilde{x}\in M \; \text{ and } 0<r<1. 
	\end{equation}

	\section{Proof of Theorem \ref{Thm}}
	As mentioned earlier, our goal is to derive results for $M$, analogous to the ones obtained in (\cite{GIUMA}) for the covering space $X$.
	
	For that, in this section we study the boundedness of the maximal operator 
	\begin{equation}\label{maximal}
		\widehat{S}_{\ast}^zf(x)=\sup\limits_{R\geq \rho^2}|\widehat{S}_R^zf(x)|, \quad \text{for }f\in L^p(M), \quad 1\leq p\leq 2. 
	\end{equation}
	
	Our aim is to prove that $\widehat{S}_{\ast}^z$ maps $L^1(M)$ to $(L^{1,w}+L^{r})(M)$, for all $r=r(M)$ large enough, which is our main result in Theorem \ref{thmmapping}. 
	Then, taking into account the $L^2(M)$ boundedness result of Lemma \ref{S*L2}, Theorem \ref{Thm} follows by intepolation and well-known measure theoretic arguments, see for example \cite[Theorem 2.1.14]{GRAF}.
	
	The Riesz means kernel $\kappa_R^z$ on $X$ is given by 
	\begin{equation*}
		\kappa_R^z(\exp H)=\mathcal{H}^{-1}\left(\left(1-\frac{\lambda^2+\rho^2}{R}\right)^z_{+}\right)\left(\exp H\right), \quad \lambda \in \mathbb{R},\;H \in \mathbb{R}_{+}. 
	\end{equation*}
	Using this formula and the expression of the inverse spherical Fourier transform $\mathcal{H}^{-1}$ in the case of rank one symmetric spaces, Giulini and Mauceri in \cite[Corollary 3.7]{GIUMA} obtained the following estimate of $\kappa_R^z$: 
	\begin{equation}\label{estimate}
		|\kappa_R^z(\exp H)| \leq c(z)R^{n/2}(1+\sqrt{R}H)^{-\mathrm{Re}z-(n+1)/2}(1+H)^{(n-1)/2}e^{-\rho H},
	\end{equation}
	where $c(z)$ is some constant which grows at most exponentially in
	$\mathrm{Im}z$ when $\mathrm{Re}z$  is in a bounded subset of $(0,+\infty)$.
	
	Recall now that the Riesz means operator on $M$ is initially defined as a convolution operator on $G$,
	\begin{equation}\label{33}
		(\widehat{S}_R^zf)(x)=\int_G\kappa_R^z(y^{-1}x)f(y)dy, \quad f\in C_0^{\infty}(M).
	\end{equation}
	Set $\kappa_R^z(x,y) =\kappa_R^z(y^{-1}x) $ and 
	\begin{equation}\label{serieskernel}
		\widehat{\kappa}_R^z(\tilde{x},\tilde{y})=\sum_{\gamma \in \Gamma}\kappa_R^z(x,\gamma y), 
	\end{equation}
	where $\tilde{x}=\pi(x)$ and $\pi:X\rightarrow M$ denotes the covering map.
	We shall first prove the following result.
	\begin{proposition}
		If $M\in (R)$, then the series (\ref{serieskernel}) converges and the Riesz means operator $\widehat{S}^z_R$ on $M$ is given by 
		\begin{equation}\label{35}
			(\widehat{S}_R^z f) (\tilde{x}) =\int_M\widehat{\kappa}^z_R(\tilde{x},\tilde{y}) f (\tilde{y})d\tilde{y}.
		\end{equation}
	\end{proposition} 
	\begin{proof}
		Use the Cartan decomposition and write $(\gamma y)^{-1}x =k \exp H_{\gamma} k_{\gamma}^{\prime}$. Note that $d(x, \gamma y)=H_{\gamma}$. Then, since $\kappa_{R}^z$ is $K$-bi-invariant, we have $\kappa_{R}^z((\gamma y)^{-1}x) =\kappa_{R}^z(\exp H_\gamma )$. The distance on $M$ is defined by
		\begin{equation}\label{distM}
			d_M(\tilde{x},\tilde{y})=\inf\limits_{\gamma \in \Gamma} d(x,\gamma y).
		\end{equation}
		
		Recall that $\delta(\Gamma)<\rho$. Then, estimate (\ref{estimate})  implies that  for any $0<\varepsilon <\rho-\delta(\Gamma)$,
		\begin{align}
			|\widehat{\kappa}_{R}^{z}(\tilde{x},\tilde{y})|&\leq c(z)R^{n/2}\sum_{ \gamma \in \Gamma}(1+\sqrt{R}H_{\gamma})^{-\mathrm{Re}z-(n+1)/2}(1+H_{\gamma})^{(n-1)/2}e^{-\rho H_{\gamma}} \label{original} \\
			&\leq c(z)R^{n/2}\sum_{ \gamma \in \Gamma} (1+\sqrt{R}H_{\gamma})^{-\mathrm{Re}z-(n+1)/2}e^{-(\rho-\varepsilon) H_{\gamma}} \notag \\
			&\leq c(z)R^{n/2} (1+\sqrt{R}d_M(\tilde{x}, \tilde{y}))^{-\mathrm{Re}z-(n+1)/2}\sum_{ \gamma \in \Gamma}e^{-(\rho-\varepsilon) H_{\gamma}} \label{(3)}
		\end{align}
		where we used (\ref{distM}). Thus, it remains to prove (\ref{35}). Since $\kappa_{R}^z$ and $f$ are right-$K$-invariant, from (\ref{33}) we get that
		\begin{equation*}
			(\widehat{S}_R^zf)(x)=\int_{X}\kappa_R^z(x,y)f(y)dy.
		\end{equation*}
		Finally, since $f$ is left $\Gamma$-invariant, by Weyl’s formula we find that
		\begin{align*}
			(\widehat{S}_R^zf)(x)&=\int_{X}\kappa_R^z(x,y)f(y)dy=\int_{\Gamma \backslash X}\left( \sum_{ \gamma \in \Gamma} \kappa_R^z(x,\gamma y)f(\gamma y)\right)d\tilde{y}\\&=\int_M \widehat{\kappa}_R^z(\tilde{x},\tilde{y})f(\tilde{y})d\tilde{y}.
		\end{align*}
	\end{proof}

	Next, we proceed as in \cite[Lemma 4.1]{GIUMA}, and we prove the following lemma.
	\begin{lemma}\label{S*L2}  If $f\in L^2(M)$ and $z\in \mathbb{C}$ with $\mathrm{Re}z>0$,
		then $\|\widehat{S}_{\ast}^zf\|_2\leq c(z)\|f\|_2$.
	\end{lemma}
	\begin{proof}
		Let $\widehat{H}_R=e^{R\Delta_M}$ be the heat semigroup on $M$, where $\Delta_M$ is the Laplace-Beltrami operator on $M$. By the spectral theorem and the fact that the bottom of the spectrum on $M\in (R)$ is $\rho^2$, it holds $\|\widehat{H}_R\|_{L^2(M)\rightarrow L^2(M)} =e^{-\rho^2 R} \leq 1$. Thus, by \cite[Chapter III, MAXIMAL THEOREM]{STEINred}, the heat maximal operator $f\rightarrow \widehat{H}_{\ast}f:=\sup_{R>0}|\widehat{H}_Rf|$ is bounded on $L^2 (M)$. Thus, it suffices to prove the boundedness of $(\widehat{S}^z-\widehat{H})_{\ast}$. Using the spectral theorem for $\Delta_M$ and the Mellin transform we have
		\begin{equation}\label{mellin}
			(\widehat{S}^z_R-\widehat{H}_R)f=\int_{\mathbb{R}}c(z,s)R^{-is}(-\Delta_M	)^{is}fds, 
		\end{equation}
		where $|c(z,s)|\leq c(z)(1+|s|)^{-(\mathrm{Re}z+1)}$ \cite{GIUMA}. So, the integral in (\ref{mellin}) converges. Since $L^2(M)$ is a complete Banach lattice, from \cite{COW}, we can write
		\[\
		(\widehat{S}^z-\widehat{H})_{\ast}f=\sup_{R>0}|(\widehat{S}_R^z-\widehat{H}_R)f|\leq c(z) \int_{\mathbb{R}} (1+|s|)^{-(\mathrm{Re}z+1)}|(-\Delta_M)^{is}f|ds.
		\]
		Thus, since by the spectral theorem $\|(-\Delta_M)^{is}\|_{L^2(M)\rightarrow L^2(M)}\leq 1$, we obtain
		\[\
		\|(\widehat{S}^z-\widehat{H})_{\ast}f\|_{L^2(M)}\leq c(z)\|f\|_{L^2(M)}.
		\]
	\end{proof}
	Our main result is the following, corresponding to \cite[Lemma 4.2]{GIUMA}.
	
	\begin{theorem}\label{thmmapping}
		Let $\mathrm{Re}z> (n - 1)/2$ and $M\in (R)$. Consider $q_0=q_0(M)=2\rho/\varepsilon$, where $0<\varepsilon <\rho-\delta(\Gamma)$. Then, for all $1 < p \leq q_0^{\prime}$, the operator $\widehat{S}_{\ast}^z$ maps $L^p (M)$ continuously into $(L^p + L^r)(M)$ for every $r \in [q_0p^{\prime} /(p^{\prime}-q_0), + \infty]$. Moreover, $\widehat{S}_{\ast}^z$ maps $L^1(M)$ continuously into $(L^{1,w}+L^r)(M)$ for every $r\in[q_0, +\infty]$. 
	\end{theorem}
	\begin{proof}
		We have
		\begin{align*}
			\widehat{S}_R^zf(\tilde{x})&=\int_{M}\widehat{\kappa}^z_R(\tilde{x}, \tilde{y})f(\tilde{y})d\tilde{y}\\
			&=\int_{B_M(\tilde{x},1)}\widehat{\kappa}^z_R(\tilde{x}, \tilde{y})f(\tilde{y})d\tilde{y}+\int_{B_M(\tilde{x},1)^c}\widehat{\kappa}^z_R(\tilde{x}, \tilde{y})f(\tilde{y})d\tilde{y}:=I_1+I_2.
		\end{align*}
		\textit{Estimates for $I_1$.} 	Let ${N}_{\psi}$ denote the maximal operator
		\begin{equation}\label{psi} \mathcal{N}_{\psi}f(\tilde{x})=\sup\limits_{t\in (0,1)}\psi(t)\int\limits_{B_M(\tilde{x},t)} |f (\tilde{y})|d\tilde{y}, \quad \psi(t)=t^{-n}, \quad t>0.
		\end{equation}
		Recall the kernel estimate (\ref{(3)}), which holds for any $0<\varepsilon<\rho-\delta(\Gamma)$. Using the uniform estimate of Poincar{\'e} series (\ref{poincare_Cocompact}), it follows that for some $c=c(z, \rho, \varepsilon)$ that
		\begin{align}
			|I_1|&\leq \int_{B_M(\tilde{x},1)}|\widehat{\kappa}_{R}^z(\tilde{x}, \tilde{y})||f(\tilde{y})|d\tilde{y} \notag \\
			&\leq c \, R^{n/2} \int_{B_M(\tilde{x},1)}(1+\sqrt{R}d_M(\tilde{x}, \tilde{y}))^{-\mathrm{Re}z-(n+1)/2}|f(\tilde{y})|d\tilde{y}\notag\\
			&= c \int_{B_M(\tilde{x},1)}d_M(\tilde{x}, \tilde{y})^{-n} \frac{(\sqrt{R}d_M(\tilde{x}, \tilde{y}))^n}{(1+\sqrt{R}d_M(\tilde{x}, \tilde{y}))^{\mathrm{Re}z+(n+1)/2}}|f(\tilde{y})|d\tilde{y} \notag \\
			&= c \sum\limits_{\nu=-\infty}^0\int_{2^{\nu-1}<d_M(\tilde{x},\tilde{y})\leq  2^{\nu}}d_M(\tilde{x}, \tilde{y})^{-n} \frac{(\sqrt{R}d_M(\tilde{x}, \tilde{y}))^n}{(1+\sqrt{R}d_M(\tilde{x}, \tilde{y}))^{\mathrm{Re}z+(n+1)/2}}|f(\tilde{y})|d\tilde{y} \notag \\
			&\leq c \sum\limits_{\nu=-\infty}^0\int_{2^{\nu-1}<d_M(\tilde{x},\tilde{y})\leq  2^{\nu}} 2^{-n(\nu-1)}\frac{(\sqrt{R}2^{\nu})^{n}}{(1+\sqrt{R}2^{\nu-1})^{\mathrm{Re}z+(n+1)/2}}|f(\tilde{y})|d\tilde{y}\notag \\
			&\leq c \sum\limits_{\nu=-\infty}^0 \frac{(\sqrt{R}2^{\nu})^{n}}{(1+\sqrt{R}2^{\nu-1})^{\mathrm{Re}z+(n+1)/2}}\psi(2^{\nu})\int_{d_M(\tilde{x},\tilde{y})\leq 2^{\nu}}|f(\tilde{y})|d\tilde{y}\notag\\
			&\leq c \mathcal{N}_{\psi}f(\tilde{x})\sum\limits_{\nu=-\infty}^0 \frac{(\sqrt{R}2^{\nu})^{n}}{(1+\sqrt{R}2^{\nu-1})^{\mathrm{Re}z+(n+1)/2}}.\label{seira},
		\end{align}
		where for the last two lines we used (\ref{psi}). Let $2^k<\sqrt{R}\leq 2^{k+1}$, for some $k\in \mathbb{Z}$. Thus, $\sqrt{R}=2^{k+s}$, for some $s \in (0,1]$. Then, note that 
		\begin{align}
			\sum\limits_{\nu=-\infty}^0\frac{(\sqrt{R}2^{\nu})^{n}}{(1+\sqrt{R}2^{\nu-1})^{\mathrm{Re}z+(n+1)/2}}&=\sum\limits_{\nu=-\infty}^0\frac{(2^{k+\nu+s})^n}{(1+2^{k+\nu-1+s})^{\mathrm{Re}z+(n+1)/2}} \notag \\
			&\leq \sum\limits_{\ell=-\infty}^{+\infty}\frac{(2^{\ell+s})^n}{(1+2^{\ell-1+s})^{\mathrm{Re}z+(n+1)/2}}\notag\\ 
			& \leq c\sum\limits_{\ell=-\infty}^{0}(2^{\ell})^n+c\sum\limits_{\ell=0}^{+\infty}(2^{\ell})^{-\mathrm{Re}z+(n-1)/2}<c, \label{newsum}
		\end{align}
		provided that $\mathrm{Re}z>(n-1)/2$. Combining (\ref{seira}) and (\ref{newsum}), it follows that 
		\begin{equation}\label{local}
			|I_1|\leq c(z)\mathcal{N}_{\psi}f(\tilde{x}).
		\end{equation}
		By (\ref{small_balls}), we have that $\psi(t)$  behaves like $|B_M(\tilde{x},t)|^{-1}$ as $t\rightarrow 0^+$. Thus, $|I_1|$ is dominated by the Hardy-Littlewood maximal function over the balls of radius at most $1$. A standard covering lemma shows that the maximal operator $\mathcal{N}_{\psi}$ is of weak type $1-1$ and is bounded on $L^p (M)$ for every $1<p\leq \infty$.
		
		\textit{Estimates for $I_2$.}  Firstly, note that when $d_M(\tilde{x}, \tilde{y})>1$, we have $d(x,\gamma y)>1$ for all $\gamma \in \Gamma$. Then, we can control  $(1+\sqrt{R}d(x,\gamma y))^{-1}$ above by $2 (1+\sqrt{R})^{-1}(1+d(x,\gamma y))^{-1}$ if $R\geq 1$, and by $\sqrt{R}^{-1}(1+d(x,\gamma y))^{-1}$ if $R< 1$. Therefore, in both cases, when $\mathrm{Re}z>(n-1)/2$, the kernel estimate (\ref{original}) yields
		\begin{align*}
			|\widehat{\kappa}_{R}^z(\tilde{x}, \tilde{y})|&\leq c(z)\sum_{ \gamma \in \Gamma}(1+d(x,\gamma y))^{-\mathrm{Re}z-1}e^{-\rho d(x,\gamma y)}\\
			&\leq c(z)(1+d_M(\tilde{x},\tilde{y}))^{-(n+1)/2}\sum_{ \gamma \in \Gamma}e^{-\rho d(x,\gamma y)}, \quad d_M(\tilde{x}, \tilde{y})>1.
		\end{align*}
		Therefore, 		
		\begin{align*}
			|I_2|&\leq \int_{B_M(\tilde{x},1)^c}|\widehat{\kappa}_{R}^z(\tilde{x}, \tilde{y})||f(\tilde{y})|d\tilde{y} \\&\leq c(z)\int_{B_M(\tilde{x},1)^c}(1+d_M(\tilde{x},\tilde{y}))^{-(n+1)/2}\sum_{ \gamma \in \Gamma}e^{-\rho d(x,\gamma y)}|f(\tilde{y})|d\tilde{y}
			\\&\leq c(z)\int_{M}\widehat{\kappa}(\tilde{x}, \tilde{y})|f(\tilde{y})|d\tilde{y},
		\end{align*}
		where 
		\[ \widehat{\kappa}(\tilde{x}, \tilde{y})=(1+d_M(\tilde{x}, \tilde{y}))^{-(n+1)/2}\sum_{ \gamma \in \Gamma}e^{-\rho d(x,\gamma y)}, \quad \tilde{x}, \tilde{y}\in M.
		\]
		Note that this kernel is uniformly bounded on $M\times M$ by (\ref{poincare_Cocompact}). Thus, $|I_2|$ is dominated by an integral operator with kernel $\widehat{\kappa}$, acting on $|f|$. To conclude, it suffices to show that for all $q$ large enough, 
		\[
		\sup_{\tilde{x}\in M}\|\widehat{\kappa}(\tilde{x}, \cdot)\|_{L^q(M)}<+\infty \quad \text{and} \quad \sup_{\tilde{y}\in M}\|\widehat{\kappa}(\cdot, \tilde{y})\|_{L^q(M)}<+\infty.
		\]
		By symmetry, we may restrict to the first norm. Then, for every $0<\varepsilon <\rho-\delta (\Gamma)$, we have
		\begin{align*}
			&\int_M \widehat{\kappa}^q(\tilde{x}, \tilde{y})d\tilde{y}=\int_M (1+d_M(\tilde{x}, \tilde{y}))^{-q(n+1)/2}\left(\sum_{ \gamma \in \Gamma}e^{-\rho d(x,\gamma y)}\right)^q d\tilde{y}\\
			\leq&\int_M (1+d_M(\tilde{x}, \tilde{y}))^{-q(n+1)/2}e^{-q\varepsilon d_M(\tilde{x}, \tilde{y})}\left(\sum_{ \gamma \in \Gamma}e^{-(\rho-\varepsilon)d(x,\gamma y)}\right)^q d\tilde{y}\\
			\leq & C^q(\rho-\varepsilon) \int_M (1+d_M(\tilde{x}, \tilde{y}))^{-q(n+1)/2} e^{-q\varepsilon d_M(\tilde{x}, \tilde{y})} d\tilde{y},
		\end{align*}
		by the uniform bound (\ref{poincare_Cocompact}). Therefore, using the volume growth estimate (\ref{volume_annulus}), the last integral can be estimated by 
		\begin{align*}
			&	\sum_{\nu=0}^{+\infty} \int_{\nu <d_M(\tilde{x}, \tilde{y})\leq \nu+1} (1+d_M(\tilde{x}, \tilde{y}))^{-q(n+1)/2} e^{-q\varepsilon d_M(\tilde{x}, \tilde{y})} d\tilde{y}\\
			\leq
			c\, &\sum_{\nu=0}^{+\infty}(1+\nu)^{-q(n+1)/2}(1+\nu)^{n}e^{-q\varepsilon\nu} e^{2\rho \nu},
		\end{align*}
		which is finite, if $q\geq q_0(M)=2\rho/\varepsilon  $. An application of Young's inequality finishes the proof.
	\end{proof}
	
	\textbf{Remark.} The result at infinity for $\widehat{S}_{\ast}^z$ is less precise than the one in \cite[Lemma 4.2]{GIUMA} for $S_{\ast}^z$, where $q\geq 2$ in the rank one case.

	Using complex interpolation, we have the following result.
	\begin{theorem}\label{interpol}  Let $M\in (R)$ and consider $q_0=q_0(M)=2\rho/\varepsilon$, where $0<\varepsilon<\rho-\delta(\Gamma)$. Then, for all $1\leq p\leq 2$ and $\mathrm{Re}z>(n-1)(\frac{1}{p}-\frac{1}{2})$, the following mapping properties of $\widehat{S}_{\ast}^z$ hold:  for every $f\in L^p(M)$, $1< p\leq 2$, 
		\[
		\|\widehat{S}_{\ast}^zf\|_{(L^p+L^r)(M)}\leq c(z)\|f\|_{L^p(M)},
		\] 	
		for all $r \in [pq_0/(2-p+pq_0-q_0), \infty]$. For every $f\in L^1(M)$, it holds
		\[
		\|\widehat{S}_{\ast}^zf\|_{(L^{1,w}+L^r)(M)}\leq c(z)\|f\|_{L^1(M)},
		\] 	
		for every $r\in [q_0, \infty]$.
	\end{theorem}
	
	As a corollary of Theorem \ref{interpol} and standard measure-theoretic arguments, Theorem \ref{Thm} follows.
	
	\section*{Acknowledgment}
	The author would like to thank the referee for his/her critical comments and the careful and insightful review, as well as  M. Kolountzakis and M. Papadimitrakis for conversations and remarks.

	\textbf{Data availability statement} Data sharing not applicable to this article as no datasets were generated or analyzed during the current study.
	
	\textbf{Declarations}
	
	Funding: Supported by the Hellenic Foundation for Research and Innovation, Project HFRI-FM17-1733.
	
	Conflicts of interest/Competing interests: Not applicable
	
	Availability of data and material: Not applicable
	
	Code availability: Not applicable
	
	Ethics approval: Not applicable
	
	Consent to participate: Not applicable 
	
	Consent for publication: Not applicable
	
\end{document}